\documentclass[english]{article}
\usepackage[T1]{fontenc}
\usepackage[latin9]{inputenc}
\usepackage{amsmath}
\usepackage{amsthm}
\usepackage{amssymb}

\makeatletter
\theoremstyle{plain}
\newtheorem{thm}{\protect\theoremname}
\theoremstyle{plain}
\newtheorem{question}[thm]{\protect\questionname}
\theoremstyle{plain}
\newtheorem{lem}[thm]{\protect\lemmaname}
\theoremstyle{remark}
\newtheorem{rem}[thm]{\protect\remarkname}
\theoremstyle{plain}
\newtheorem{cor}[thm]{\protect\corollaryname}

\makeatother

\usepackage{babel}
\providecommand{\corollaryname}{Corollary}
\providecommand{\lemmaname}{Lemma}
\providecommand{\questionname}{Question}
\providecommand{\remarkname}{Remark}
\providecommand{\theoremname}{Theorem}

\begin{document}
\title{Random centers of localization for random operators}
\author{Raphael Ducatez}
\maketitle
\begin{abstract}
We propose a new random process to construct the eigenvectors of some
random operators which make a short and clean connection with the
resolvent. In this process the center of localization has to be chosen
randomly.
\end{abstract}

\section{Introduction}

We consider a matrix $H\in\mathbb{R}^{|\Lambda|\times|\Lambda|}$,
$\Lambda$ a finite set, defined as
\[
H=T+V
\]
where $T$ is a fixed symmetric matrix and $V=\text{diag}((v_{x})_{x\in\Lambda})$
is a diagonal matrix whose entries $(v_{x})$ are random, independent
and with law $\rho_{x}$ and we are interested in the following question. 
\begin{question}
\label{que:motivation}Assuming $\lambda$ is an eigenvalue of $H$
what does the eigenvector $\phi_{\lambda}$ look like ? 
\end{question}

Such a question were first asked by P. Anderson \cite{anderson1958absence}
who predicted that in the tight binding model with high disorder the
eigenvectors should be localized and since then the subject has been
very active field of research in physics. A rigorous proof of the
now called Anderson localization has been first given in \cite{kunz1980spectre}
for the random Schrodinger operator in one dimension and then in \cite{frohlich1983absence,aizenman1993localization}
in any dimension with high disorder. See also the books \cite{carmona2012spectral}
and \cite{aizenman2015random} for an introduction of the topics and
the many related questions (dynamical localization, local statistic
of the eigenvalues, delocalization or trees,...). 

The present paper is motivated by the following remarks : In the literature
the mechanics of the proofs heavily relied on the resolvent $(H-z)^{-1}$,
$z\in\mathbb{C}$. The reason is that one can use the resolvent formula,
which allows a lot of algebraic computations to obtain very good estimates
on $(H-z)^{-1}$. On the others hand the results on the eigenvectors
are mostly qualitative. The difficulty here is that eigenvalue $\lambda$
is random itself and then the associated eigenvector $\phi_{\lambda}$
is not as easy to manipulate as the resolvent. In these notes, our
main result is to give a new random process to construct the eigenvector
which make a clear and direct connection with the resolvent. The key
input is to choose the center of localization randomly. Such a process
already appeared in \cite{rifkind2018eigenvectors} and \cite{ducatez2019forward}
but in a less general statement. 

.

\section{Two random constructions of the eigenvectors}

We assume that $T$ and $V$ are such that 
\begin{enumerate}
\item The law of the $(v_{x})$ are absolutely continuous,
\item $\text{Spec}(H)\subset[0,1]$ a.s,
\item The eigenvalues are non degenerate a.s.
\end{enumerate}
For the second condition one should just think $T$ and $V$ as bounded
matrices that has been re-scaled for convenience. The third point
actually follows from Minami estimate \cite{minami1996local}. We
add it just to make sure that the definition of $\phi_{\lambda}$
the eigenvector associated to an eigenvalue $\lambda$ is not ambiguous
(up to a phase). Then the first condition is by far the most critical
one.

We propose two random processes on $\Omega=\mathbb{R}^{\Lambda}\times[0,1]\times\mathbb{S}^{\Lambda-1}\times\Lambda$
that we will denote $\mu_{1}$ and $\mu_{2}$ and where we denote
$\mathbb{S}^{\Lambda-1}=\{u\in\mathbb{R}^{\Lambda}:\|u\|_{L^{2}}=1\}$.

\subsubsection*{Construction of $\mu_{1}$ : }

The law $\mu_{1}$ should be seen as the standard way to construct
an eigenvector and it is the process we are interested in.
\begin{enumerate}
\item Draw the random diagonal $V$ (with law $\otimes_{x\in\Lambda}\rho_{x}$)
\item Diagonalize $H=T+V$, choose a random eigenvalue $\lambda\in\text{Spec}(H)$
uniformly and denote $\phi_{\lambda}$ its corresponding eigenvector. 
\item Choose a random point $x^{*}\in\Lambda$ with the conditional law
$\mathbb{P}(x^{*}=x|V,\lambda,\phi_{\lambda})=|\phi_{\lambda}(x)|^{2}$.
We will call $x^{*}$ a \emph{random center} of $\phi_{\lambda}$.
\end{enumerate}
We have then obtained $(V,\lambda,\phi_{\lambda},x^{*})\in\Omega$
and $\mu_{1}$ is the law of this construction. Remark that if $\phi_{\lambda}$
is localized, the point $x^{*}$ should be a good guess of its domain
of localization, which is the reason we also refer $x^{*}$ as the
\emph{center of localization} of $\phi_{\lambda}$. 

We now construct the process $\mu_{2}$, but first we will need the
following very standard Lemma (see for example \cite[Theorem 5.3]{aizenman2015random}).
For $x\in\Lambda$, we denote $H_{v_{x}=0}$ such that $H=H_{v_{x}=0}+v_{x}1_{x}1_{x}^{*}$.
\begin{lem}
\label{lem:One-to-One}For any $\lambda\in\mathbb{R}$, $x\in\Lambda$
and $(v_{y})\in\mathbb{R}^{\Lambda\setminus\{x\}}$ there exists at
most a unique $v_{x}\in\mathbb{R}$ such that $\lambda$ is an eigenvalue
of $T+V$ with $\phi_{\lambda}(x)\neq0$. Moreover 
\[
v_{x}=-\frac{1}{(H_{v_{x}=0}-\lambda)_{xx}^{-1}}\quad\text{and}\quad\phi_{\lambda}=\frac{(H_{v_{x}=0}-\lambda)^{-1}1_{x}}{\|(H{}_{v_{x}=0}-\lambda)^{-1}1_{x}\|}.
\]
\end{lem}

\begin{proof}
[Proof of Lemma \ref{lem:One-to-One}]If such a $v_{x}$ exists, we
have $(H_{v_{x}=0}-\lambda+v_{x}1_{x}1_{x}^{*})\phi_{\lambda}=0$
and then 
\[
(H_{v_{x}=0}-\lambda)\phi_{\lambda}=-v_{x}\phi_{\lambda}(x)1_{x}
\]
so that we have 
\[
\phi_{\lambda}=-v_{x}\phi_{\lambda}(x)(H{}_{v_{x}=0}-\lambda)^{-1}1_{x}.
\]
Here $v_{x}\phi_{\lambda}(x)\in\mathbb{R}$ and can be obtained using
the normalization $\|\phi_{\lambda}\|=1$. Moreover, $\phi_{\lambda}(x)=-v_{x}\phi_{\lambda}(x)(H{}_{v_{x}=0}-\lambda)_{xx}^{-1}$
which finishes the proof of the Lemma.
\end{proof}

\subsubsection*{Construction of $\mu_{2}$ : }

The second process $\mu_{2}$ is also natural and has a clear connection
with the resolvent. It will be a very useful tool for the study of
$\mu_{1}$.
\begin{enumerate}
\item Draw $\lambda\in[0,1]$ randomly with uniform law. 
\item Choose a random point $x^{*}\in\Lambda$ with uniform law. 
\item Draw the random diagonal entries of $V$ except for $v_{x^{*}}$ (with
law $\otimes_{y\in\Lambda\setminus\{x^{*}\}}\rho_{y}$).
\item Construct $v_{x^{*}}$ and $\phi_{\lambda}$ as in Lemma \ref{lem:One-to-One}
such that $\lambda$ is an eigenvalue of $H=T+V$.
\end{enumerate}
We then obtain $((v_{y})_{y\neq x^{*}},v_{x^{*}}),\lambda,\phi_{\lambda},x^{*})\in\Omega$
and denote $\mu_{2}$ the law of this construction. We can now state
the main result of this paper.
\begin{thm}
\label{lem:Radon-Nikodym}We have the following Radon-Nikodym derivative
\[
\frac{d\mu_{1}}{d\mu_{2}}(V,\lambda,\phi_{\lambda},x^{*})=\frac{d\rho_{x^{*}}}{dv_{x^{*}}}(v_{x^{*}})
\]
\end{thm}

\begin{proof}
[Proof of Theorem \ref{lem:Radon-Nikodym}]Let $f:\Omega\rightarrow\mathbb{R}$
a test function and denoting $\omega=(V,\lambda,\phi_{\lambda},x^{*})$
we have 
\begin{align*}
\mathbb{E}_{\mu_{1}}(f(\omega)) & =\frac{1}{|\Lambda|}\int_{\mathbb{R}^{\Lambda}}\sum_{\lambda\in\text{Spec}(H)}\sum_{x^{*}\in\Lambda}f(\omega)|\phi_{\lambda}(x^{*})|^{2}\prod_{y\in\Lambda}d\rho_{y}(v_{y})\\
 & =\frac{1}{|\Lambda|}\sum_{x^{*}\in\Lambda}\sum_{i=1}^{|\Lambda|}\int_{\mathbb{R}^{\Lambda\setminus\{x^{*}\}}}\int_{\mathbb{R}}f(\omega)|\phi_{\lambda_{i}}(x^{*})|^{2}d\rho_{x^{*}}(v_{x^{*}})\prod_{y\in\Lambda\setminus\{x^{*}\}}d\rho_{y}(v_{y})
\end{align*}
where we denote $\lambda_{1}\leq\cdots\leq\lambda_{|\Lambda|}$ the
eigenvalues of $H$. We also denote $\lambda_{1}^{(x^{*})}\leq\cdots\leq\lambda_{|\Lambda|-1}^{(x^{*})}$
the eigenvalues of the restricted matrix $H|_{\Lambda\setminus\{x^{*}\}}$
and by interlacing we have $\lambda_{i}\in[\lambda_{i-1}^{(x^{*})},\lambda_{i}^{(x^{*})}]$
for all $1\leq i\leq|\Lambda|$ where $\lambda_{0}^{(x^{*})}=0$ and
$\lambda_{|\Lambda|}^{(x^{*})}=1$. With $(v_{y})_{y\neq x^{*}}$
fixed, the eigenvalues can be seen as functions of $v_{x^{*}}$ :
$\lambda_{i}=\lambda_{i}(v_{x^{*}})$ and we make the following change
of variable $v_{x^{*}}\rightarrow\lambda_{i}$. Because $\frac{\partial\lambda_{i}}{\partial v_{x^{*}}}=|\phi_{\lambda_{i}}(x^{*})|^{2}$
we obtain 
\begin{align*}
\mathbb{E}_{\mu_{1}}(f(\omega)) & =\frac{1}{|\Lambda|}\sum_{x^{*}\in\Lambda}\sum_{i=1}^{|\Lambda|}\int_{\mathbb{R}^{\Lambda\setminus\{x^{*}\}}}\left(\int_{\lambda_{i-1}^{(x^{*})}}^{\lambda_{i}^{(x^{*})}}f(\omega)\frac{d\rho_{x^{*}}}{dv_{x^{*}}}d\lambda\right)\prod_{y\in\Lambda\setminus\{x^{*}\}}d\rho_{y}(v_{y}).\\
 & =\frac{1}{|\Lambda|}\sum_{x^{*}\in\Lambda}\int_{\mathbb{R}^{\Lambda\setminus\{x^{*}\}}}\int_{[0,1]}f(\omega)\frac{d\rho_{x^{*}}}{dv_{x^{*}}}d\lambda\prod_{y\in\Lambda\setminus\{x^{*}\}}d\rho_{y}(v_{y})\\
 & =\mathbb{E}_{\mu_{2}}\left(f(\omega)\frac{d\rho_{x^{*}}}{dv_{x^{*}}}\right).
\end{align*}
\end{proof}
\begin{rem}
The hypothesis the entries $(v_{x})$ are independent can be removed
if we replace $\frac{d\rho_{x^{*}}}{dv_{x^{*}}}$ by the law of $v_{x^{*}}$
conditionally on $(v_{y})_{y\neq x^{*}}$. The proof is exactly the
same in that case.
\end{rem}

Theorem \ref{lem:Radon-Nikodym} propose an interesting answer to
Question \ref{que:motivation}. Indeed, conditionally on $\lambda$,
steps 2-4 in the construction of $\mu_{2}$ gives a recipe to construct
$\phi_{\lambda}$. The answer here is that we have a direct link between
the resolvent and the eigenvector but with one particular feature
: in order to be able to use Lemma \ref{lem:One-to-One} the starting
point $x^{*}\in\Lambda$ has to be chosen randomly. 
\begin{rem}
We denote $\|\rho\|_{\infty}=\sup_{x\in\Lambda}\|\frac{d\rho_{x}}{dv_{x}}\|_{\infty}$.
Then for any $A\subset\Omega$
\[
\mu_{1}(A)\leq\|\rho\|_{\infty}\mu_{2}(A).
\]
In particular, if $\|\rho\|_{\infty}<\infty$, any event $A$ that
occurs with high probability for $\mu_{2}$ (ie $\mu_{2}(\Omega\setminus A)\ll1$)
occurs with high probability for $\mu_{1}$.
\end{rem}

\section{Applications}

We now state a few consequences of Theorem \ref{lem:Radon-Nikodym}. 

\subsection{Law of the eigenvalues and eigenvectors}

We denote $\mathbb{E}_{\neq x}$ the mean on all the entries of $V$
but $x$ and similarly $\mathbb{P}_{\neq x}$ the probability
\[
\mathbb{E}_{\neq x}(\cdots)=\int_{\mathbb{R}^{\Lambda\setminus\{x\}}}\cdots\prod_{y\ne x}d\rho(v_{y})\quad\text{and}\quad\mathbb{P}_{\neq x}(A)=\mathbb{E}_{\neq x}(1_{A})\text{ for all }A.
\]

The first point of the following Corollary is a standard result which
is usually proved considering the limit $\lim_{\epsilon\rightarrow0^{+}}\mathbb{E}\left(\frac{1}{|\Lambda|}\Im\text{Tr}((H-\lambda+i\epsilon)^{-1})\right)$
(see for example \cite[Chapter 4]{aizenman2015random}). We claim
that it is just a particular case of Theorem \ref{lem:Radon-Nikodym}
which is much more general because it also gives the law for the eigenvector.
\begin{cor}
\label{cor:Eigenvector-law}
\begin{enumerate}
\item The density of state $\nu$ is given by 
\[
\frac{d\nu}{d\lambda}(\lambda)=\frac{1}{|\Lambda|}\sum_{x\in\Lambda}\mathbb{E}_{\neq x}\left[\frac{d\rho_{x}}{dv_{x}}\left(\frac{1}{(\lambda-H{}_{v_{x}=0})_{xx}^{-1}}\right)\right]
\]
In particular $\|\frac{d\nu}{d\lambda}\|_{\infty}\leq\|\rho\|_{\infty}$
(Wegner Estimate).
\item For any subset $U\subset\mathbb{S}^{\Lambda-1}$ 
\begin{align*}
 & \mathbb{P}(\phi_{\lambda}\in U|\lambda\in\text{Spec}(H))\\
 & =\frac{1}{|\Lambda|\frac{d\nu}{d\lambda}(\lambda)}\sum_{x\in\Lambda}\mathbb{E}_{\neq x}\left[1_{U}\left(\frac{(H_{v_{x}=0}-\lambda)^{-1}1_{x}}{\|(H{}_{v_{x}=0}-\lambda)^{-1}1_{x}\|}\right)\frac{d\rho_{x}}{dv_{x}}\left(\frac{1}{(\lambda-H{}_{v_{x}=0})_{xx}^{-1}}\right)\right]
\end{align*}
\end{enumerate}
\end{cor}

\begin{proof}
[Proof of Corollary \ref{cor:Eigenvector-law}]With a test function
that only depends on $\lambda$, $f(\omega)=f(\lambda)$ we have
\begin{align*}
\mathbb{E}_{\mu_{1}}(f(\lambda)) & =\mathbb{E}_{\mu_{2}}\left(f(\lambda)\frac{d\rho_{x^{*}}}{dv_{x^{*}}}\right)\\
 & =\int_{0}^{1}f(\lambda)\left(\frac{1}{|\Lambda|}\sum_{x\in\Lambda}\mathbb{E}_{\neq x}\left[\frac{d\rho_{x}}{dv_{x}}\left(\frac{1}{(\lambda-H{}_{v_{x}=0})_{xx}^{-1}}\right)\right]\right)d\lambda
\end{align*}
and we deduce the first point of the Corollary. The second point follows
similarly 
\begin{align*}
 & \mathbb{E}_{\mu_{1}}(f(\lambda)1_{\phi_{\lambda}\in U})\\
 & =\mathbb{E}_{\mu_{2}}\left(f(\lambda)1_{\phi_{\lambda}\in U}\frac{d\rho_{x^{*}}}{dv_{x^{*}}}\right)\\
 & =\int_{0}^{1}f(\lambda)\left(\frac{1}{|\Lambda|}\sum_{x\in\Lambda}\mathbb{E}_{\neq x}\left[1_{U}\left(\frac{(H_{v_{x}=0}-\lambda)^{-1}1_{x}}{\|(H{}_{v_{x}=0}-\lambda)^{-1}1_{x}\|}\right)\frac{d\rho_{x}}{dv_{x}}\left(\frac{1}{(\lambda-H{}_{v_{x}=0})_{xx}^{-1}}\right)\right]\right)d\lambda\\
 & =\int_{0}^{1}f(\lambda)\mathbb{E}_{\mu_{1}}\left(1_{\phi_{\lambda}\in U}|\lambda\right)\frac{d\nu}{d\lambda}(\lambda)d\lambda
\end{align*}
\end{proof}

\subsection{From the localization of the resolvent to the localization of the
eigenvector}

The resolvent is by far the main tool to study the Anderson model
and most of the computation and estimate is done on this object. Because
Theorem \ref{lem:Radon-Nikodym} gives a direct connection between
the law of the eigenvector and the resolvent one have a very short
proof to deduce ``eigenvector localization'' from the localization
of the resolvent. We introduce the following localization event
\[
U_{\text{loc}}^{\eta}=\{\phi\in\mathbb{S}^{\Lambda-1}:\exists x\in\Lambda,\,\forall y\in\Lambda\,|\phi(y)|\leq\eta(x,y)\}
\]
for some function $\eta:\Lambda^{2}\rightarrow\mathbb{R}$. Such a
function describes a localization phenomena if $\eta(x,y)\rightarrow0$
when $y$ is far away from $x$. For example, in the case $\Lambda\subset\mathbb{Z}^{d}$,
one usually chooses $\eta(x,y)\sim\exp(-c|x-y|)$ or $\eta(x,y)\sim\frac{1}{(1+|x-y|)^{k}}$.
\begin{cor}
\label{cor:Localization}For $\lambda\in[0,1]$,
\[
\mathbb{P}(\phi_{\lambda}\notin U_{\text{loc}}^{\eta}|\lambda\in\text{Spec}(H))\leq\left(\frac{d\nu}{d\lambda}(\lambda)\right)^{-1}\|\rho\|_{\infty}\alpha^{\eta}
\]
where 
\[
\alpha^{\eta}=\sup_{x\in\Lambda}\mathbb{P}_{\neq x}\left[\exists y\in\Lambda:|(H{}_{v_{x}=0}-\lambda)_{xy}^{-1}|>\eta(x,y)\|(H{}_{v_{x}=0}-\lambda)^{-1}1_{x}\|\right]
\]
\end{cor}

\begin{proof}
[Proof of Corollary \ref{cor:Localization}]From Corollary \ref{cor:Eigenvector-law}
we have 
\[
\mathbb{P}(\phi_{\lambda}\notin U_{\text{loc}}^{\eta}|\lambda\in\text{Spec}(H))\leq\frac{\|\rho\|_{\infty}}{\frac{d\nu}{d\lambda}(\lambda)}\sup_{x\in\Lambda}\mathbb{P}_{\neq x}\left[\frac{(H_{v_{x}=0}-\lambda)^{-1}1_{x}}{\|(H|_{v_{x}=0}-\lambda)^{-1}1_{x}\|}\notin U_{\text{loc}}^{\eta}\right]
\]
and the result follows from 
\[
\left\{ \frac{(H_{v_{x}=0}-\lambda)^{-1}1_{x}}{\|(H{}_{v_{x}=0}-\lambda)^{-1}1_{x}\|}\notin U_{\text{loc}}^{\eta}\right\} \subset\left\{ \exists y\in\Lambda:\frac{|(H_{v_{x}=0}-\lambda)_{xy}^{-1}|}{\|(H{}_{v_{x}=0}-\lambda)^{-1}1_{x}\|}>\eta(x,y)\right\} .
\]
\end{proof}
To prove localization of the eigenvector it then enough prove that
$\alpha^{\eta}$ is small. In the case of random Schrodinger operator
in dimension one usually use the product of $2\times2$ random matrices
of the form $\begin{pmatrix}\lambda-v_{x} & -1\\
1 & 0
\end{pmatrix}$ as in \cite{carmona2012spectral}. Notice that in that case $\lambda\in\mathbb{R}$
is fixed. In the case of random Schrodinger operator in any dimension
with large disorder, it can be done using the so called multi-scaled
analysis \cite{frohlich1983absence} or more directly the Fractional
Moment Method \cite{aizenman1993localization} 
\[
\mathbb{E}(|(H_{v_{x}=0}-\lambda)_{xy}^{-1}|^{s})\leq Ce^{-c|x-y|}
\]
for some $s,c,C>0$ and a Markov estimate.

\subsection{The 1-dimensional random Schrodinger operator}

In \cite{rifkind2018eigenvectors} Rifkind and Virag consider the
scaling limit of the eigenvector of the critical 1-dimensional random
Schrodinger operator 
\[
H=-\Delta+\frac{1}{\sqrt{n}}V
\]
defined on $\ell^{2}([1,n])$ where $\Delta$ is the discrete Laplacian
on this set. The authors proved the following very nice Theorem.
\begin{thm}
\cite{rifkind2018eigenvectors} As $n\rightarrow\infty$ the form
of the eigenvector behaves as
\[
\phi_{\lambda}(\lfloor nt\rfloor)^{2}\sim\frac{1}{C}\exp\left(-\tau(\lambda)\left(\frac{|t-U|}{4}+\frac{B_{t-U}}{\sqrt{2}}\right)\right)\quad\text{for }t\in[0,1]
\]
with $U$ uniform on $[0,1]$, $B$ an independent Brownian motion,
$\tau$ an explicit function and $C$ a normalizing constant.
\end{thm}

We claim that one could read this Theorem and compare it with our
Theorem \ref{lem:Radon-Nikodym} as follows
\begin{itemize}
\item $U$ is the scaling limit of the center of localization $x^{*}$ chosen
uniform on $[0,n]$
\item $\exp\left(-\tau(\lambda)\left(\frac{|t-U|}{4}+\frac{B_{t-U}}{\sqrt{2}}\right)\right)$
is the scaling limit of $f(y)=|(H_{v_{x}=0}-\lambda)_{x^{*}y}^{-1}|^{2}$
for $y\in[0,n]$
\end{itemize}
A similar generalization in the finite 1D-discrete model were also
proposed in \cite{ducatez2019forward}.

\bibliographystyle{alpha-fr}
\bibliography{Biblio_Anderson}

\end{document}